\documentclass[11pt,a4paper]{article}
\usepackage{amsmath,amsfonts,amssymb,amsthm}

\input xy  \xyoption{all}

\theoremstyle{plain}
\newtheorem{teo}{Theorem}[section]
\newtheorem{prop}[teo]{Proposition}
\newtheorem{pro}[teo]{Problem}
\newtheorem{lemma}[teo]{Lemma}
\newtheorem{cor}[teo]{Corollary}

\newtheorem*{nteo*}{\namedthmname}
\newtheorem{que}[teo]{Question}

\newenvironment{nteo}[1]
  {\newcommand\namedthmname{#1}\begin{nteo*}}
  {\end{nteo*}}
\newtheorem*{ncor*}{\namedcorname}  
\newenvironment{ncor}[1]
  {\newcommand\namedcorname{#1}\begin{ncor*}}
  {\end{ncor*}}

\theoremstyle{definition}

\newtheorem{defin}[teo]{Definition}

\newtheorem{rem}[teo]{Remark}
\newtheorem{exa}[teo]{Example}

\theoremstyle{remark}

\newtheorem{notation}[teo]{Notation}
\DeclareMathOperator{\Nc}{\mathfrak{N}}
\DeclareMathOperator{\wD}{\mathfrak{wD}}

\DeclareMathOperator{\G}{\mathfrak{G}_\delta}
\DeclareMathOperator{\GG}{\mathfrak{G}_{\delta\sigma}}
\DeclareMathOperator{\F}{\mathfrak{F}_\sigma}
\DeclareMathOperator{\FF}{\mathfrak{F}_{\sigma\delta}}

\DeclareMathOperator{\sPol}{\mathfrak{Pol}}
\DeclareMathOperator{\CMC}{\mathfrak{CmC}}
\DeclareMathOperator{\CH}{\mathfrak{Char}}
\DeclareMathOperator{\ACH}{\mathfrak{a-Char}}
\DeclareMathOperator{\tu}{t_{\us}}

\DeclareMathOperator{\ttu}{t\! t_{\us}}
\DeclareMathOperator{\sv}{s_{\vs}}
\DeclareMathOperator{\su}{s_{\us}}

\DeclareMathOperator{\diam}{diam}
\DeclareMathOperator{\supp}{\mbox{supp}}

\newcommand{\A}{\mathcal{A}}
\newcommand{\T}{\mathbb{T}}
\newcommand{\PP}{\mathbb{P}}
\newcommand{\Z}{\mathbb{Z}}
\newcommand{\Q}{\mathbb{Q}}
\newcommand{\R}{\mathbb{R}}
\newcommand{\N}{\mathbb{N}}
\newcommand{\us}{\mathbf{u}}
\newcommand{\cc}{\mathfrak{c}}
\newcommand{\vs}{\mathbf{v}}
\newcommand{\ep}{\varepsilon}
\newcommand{\SUX}{\supp_\us(x)}

\newcommand{\tTu}{\tu(\T)}
\newcommand{\ttTu}{\ttu(\T)}

\newcommand{\lb}{\left\|}
\newcommand{\rb}{\right\|}
\newcommand{\NB}{}
\begin{document}

%\begin{frontmatter}

\title{On the Borel Complexity of Characterized Subgroups
\thanks{
\endgraf
%\begin
{\sl Keywords}: characterized subgroup, compact group, circle group, 
%convergent sequence, 
sequence of characters, sequence of integers, Borel hierarchy, Polishable Subgroup
\endgraf
}
% or \MSC[2008] code \sep code (2000 is the default)
%%
%\end{keyword}
}

\author{Dikran Dikranjan\thanks{The first named author gratefully acknowledges the FY2013 Long-term visitor grant~L13710 by the Japan Society for the Promotion of Science (JSPS).} and Daniele Impieri}
%\ead{dikran.dikranjan@uniud.it} 

%\author{Daniele Impieri\corref{cor}}
%\ead{daniele.impieri@uniud.it}
%\address{Dipartimento di Matematica e Informatica, Universit\`{a} di Udine, Via delle Scienze  206, 33100 Udine, Italy.}

\maketitle

\begin{abstract}
 In a compact abelian group $X$, a subgroup $H$ is called characterized  if there exists a sequence of characters $\vs=(v_n)$ of $X$ such that $H=\{x\in X:v_n(x)\to 0 \text{ in } \T\}$. Gabriyelyan proved for $X=\T$, that the characterized subgroup \NB $\{x\in\T:n!x\to 0 \text{ in }\T\}$ is not an $F_\sigma$-set. In this paper, we obtain a complete description of the $F_\sigma$-subgroups of $\T$ characterized by sequences of integers $\vs=(v_n)$ such that $v_n|v_{n+1}$ for all $n\in\N$ showing that these are exactly the countable \NB \  subgroups of $\T$. Moreover, in the general setting of compact metrizable abelian groups, we give a new point of view to study the Borel complexity of characterized subgroups in terms of appropriate test-topologies defined on the group $X$.
\end{abstract}

%\begin{frontmatter}
%%
%\begin{keyword}
%Characterized Subgroup \sep Compact Group \sep Circle Group \sep Convergent Sequence, Sequence of Characters, Sequence of integers, Borel Hierarchy, Polishable Subgroup
%
%\MSC[2010] 22C05 \sep 43A46 \sep 54H05
%%% or \MSC[2008] code \sep code (2000 is the default)
%%
%\end{keyword}
%
%\end{frontmatter}

%\pagebreak
  
We denote by $\T=\R/\Z$ the additive circle group and by $\varphi: \R \to \T$ is the canonical homomorphism; furthermore, 
we identify an element $x\in\T$ with its unique pre-image in $[0,1)$ under $\varphi$. We use $\lb \ \rb$ for the standard norm on the torus; namely, if $x\in\T$, then $\lb x\rb=x$ if $x<\frac{1}{2}$,  otherwise $\lb x\rb=1-x$. The metric generated by $\| \ \|$ will be denoted by $d$. 

For an abelian group $G$ we denote by $t(G)$ the torsion subgroup of $G$.  The  closure of a subset $M$ of a topological space will be denoted by $\overline{M}$.   

For a topological abelian group $X$ a \emph{character} of $X$ is a continuous homomorphism $\chi:X\to\T$. 
	Denote by $\widehat{X}$ the group of all characters of $X$, that is the Pontryagin dual of $X$.

\section{Introduction}

  \subsection{Some History and Basic Definitions}
		
\begin{defin} Let $\us=(u_n)_{n\in\N}$ be a sequence of integers, $X$ be a compact abelian group and $\vs=(v_n)_{n\in\N}$ be a sequence of characters of $X$.  The following subsets of $X$ 
are subgroups of $X$:
	  \begin{equation*}
	  \tu(X)=\left\{x\in X:u_nx\to0 \text{ in } X \right\} \text{ and }  \sv(X):=\left\{x\in X:v_n(x)\to0 \text{ in }\T\right\}.
	  \end{equation*}
	\end{defin}

	 Since $\widehat{\T}=\Z$, the two notions defined above coincide in the case of $X=\T$, i.e $\tTu=\sv(\T)$ where for all $n\in\N, v_n:x\mapsto u_nx$.	
	
	\begin{defin} {\rm (\cite{DMT})} Let $X$ be a compact abelian group. A subgroup $H$ of $X$ is called \emph{characterized} if there exists a sequence of characters $\vs$ such that $H=\sv(X)$. We also say that $\vs$ \emph{characterizes} $H$. 
	\end{defin}
	 
Let $\CH(X)$ be the family of all characterized subgroups of $X$. The assignment $\widehat{X}^\N \ni \us \mapsto {\su}(X) \in \CH(X)$ gives rise to a relevant equivalence relation $\sim$ in $\widehat{X}^\N$ defined by $\us \sim \vs$ when $\su(X) = \sv(X)$. 

Since $\widehat \T = \Z$, $H\le\T$ is characterized if and only if there exists $\us\in\Z^{\N}$ such that $H=\tTu$. 
The sequence $(q_n)$ of ratios of $\us$ is defined by $q_n=\frac{u_n}{u_{n-1}}$ for $n>0$ and $q_0=u_{0}$. 
If $\us$ does not have any constant subsequences there exists  a strictly increasing sequence of non-negative integers $\us^* \sim \us$ (see \cite[Proposition 2.5]{BDMW1}).  

The subgroup $\tau_\us(\R) := \varphi^{-1}(\tTu)$ can be conveniently used to study $\tTu$. This kind of subgroups of $\R$ and their subsets were studied much earlier, going back to Arbault \cite{Arb} in relation to trigonometric series and then by Borel \cite{Bo1,Bo2}. Some steps toward this direction can be found in the past in \cite{Egg}, where 
following theorem can be found: 

\begin{teo}[Eggleston]\label{TEgg}
  Let $\us=(u_n)$ be a sequence of integers.
  \begin{itemize}
   \item If the sequence of ratios $(q_n)$
   %\frac{u_n}{u_{n-1}})$ 
   is bounded, then $\tTu$ is countable.
   \item If the sequence of ratios  $(q_n)$
   %$(\frac{u_n}{u_{n-1}})$
    converges to infinity, then $\tTu$ is uncountable.
  \end{itemize}
\end{teo}

The subgroup $  \tu(X)$ appeared also as a generalization of \emph{torsion subgroup}  and \emph{$p$-torsion subgroup} of an abelian group for a prime number $p$. Armacost in \cite{Arm}, using different names, introduced for a topological abelian group the notions 
of \emph{topologically torsion subgroup} that is the subgroup of all elements $x$ such that $n!x\to0$ and \emph{topologically $p$-torsion subgroup} that is the subgroup of all elements $x$ such that $p^nx\to0$. Armacost  \cite{Arm} gave a description of $\tTu$ for $\us=(p^n)$ for a prime $p$ and posed the problem to describe $\tTu$ for $\us=(n!)$. These two notions were unified in \cite{DPS} by a special class of sequences of integers that will play an important role for this paper as well: 

\begin{defin} 
  A strictly increasing sequence of integers $\us = (u_n)$ is called {\em arithmetic} (or briefly, an {\em a-sequence}) if $u_n|u_{n+1}$ for all $n\in\N$. 
  \end{defin}

 For $X=\T$, let $\ACH(\T)$ be the class of all subgroups of $\T$ that are characterized by a-sequences.

The first results on $\ACH(\T)$ can be found already in  \cite[\S 4.4.2]{DPS}, providing, among others a partial answer to 
Armacost problem. A complete solution of Armacost problem was given in  \cite{DdS}, along with an attempt to describe $\ACH(\T)$. The gap
left in \cite{DdS} was filled in \cite{DI}, where a complete description of $\ACH(\T)$
can be found in (see \S \ref{secaseq}). 
%We show in Theorem F that as far as subgroups characterized by a-sequences are concerned, one can prove much more than the mere non-inclusion $\CH(\T)  \not \subseteq  \F(\T)$ established in Theorem \ref{teoDG1}. Namely, \emph{the only proper $F_\sigma$-subgroups of $\T$, characterized by an a-sequence are the countable ones}. Still for a-sequences we will find an explicit subgroup that witnesses the separability of $(\tTu,\pi(\tTu))$, namely the torsion subgroup $t(\tTu)$. %Moreover for $X=\T$ we will completely 

Some relevant results related to $\CH(\T)$ in the setting of number theory were obtained in \cite{B,BDS}.
The notion of a characterized subgroup was generalized, from the case of subgroups of $\T$  to the general case of topological abelian groups in \cite{DMT}.
 For more recent contributions to this field see \cite{DK,DG,DGT,BDMW1,BDMW,DH,Ga4,Ga1,Ga3,Ga2,N}. 
 
\subsection{About Borel Complexity of Characterized Subgroups}

 The following general problem has been studied in \cite{DG}.

\begin{pro}\label{probBor} Study the Borel complexity of $\sv(X)$ for a compact abelian group $X$.
\end{pro}

 \begin{notation}\label{nsub} Let us recall the first six classes of the Borel hierarchy.
  \begin{equation*}
	\begin{matrix}
		& \{\text{open sets}\} &\kern-0.5em\subseteq & \kern-0.5em\G=\{G_\delta\text{-sets}\} &\kern-0.5em\subseteq &\kern-0.5em\GG=\{\text{countable unions of $G_\delta$-sets}\}&\\
		& \{\text{closed sets}\} &\kern-0.5em \subseteq & \kern-0.5em\F=\{F_\sigma\text{-sets}\}&\kern-0.5em\subseteq &\kern-0.5em\FF=\{\text{countable intersections of $F_\sigma$-sets}\}&
	\end{matrix}
\end{equation*}
 
 Note that if $X$ is metrizable, then the diagonal inclusions hold too. More precisely, \NB 
 $$
 \{\text{open sets}\}\subseteq\F\subseteq\GG\ \ \mbox{ and }\ \ \{\text{closed sets}\}\subseteq\G\subseteq\FF.
 $$ 
 Moreover, if $X$ is infinite, then all the inclusions listed above are proper. In this paper we are interested in Borel subgroups, hence we shall use the above notation for the relative class of subgroups instead of sets. By Borel complexity of characterized subgroups we mean the study of the classes of the Borel hierarchy which these groups belong to.
\end{notation}

The following remark provides an upper bound for the Borel complexity of characterized subgroups.

\begin{rem}\label{rem1} Every characterized subgroup $H$  of a compact abelian group $X$ is an $F_{\sigma\delta}$-subgroup of $X$, and hence $H$ is a Borel subset of $X$. Indeed, if $H$ is characterized by a sequence $\vs$, this fact directly follows from the equality
  \begin{equation*}
  \sv(X) = \bigcap_{0< M<\omega} \bigcup_{m} \left( \bigcap_{n>m} S_{n,M} \right), \mbox{ where }  S_{n,M} := \left\{ x\in X : \lb v_n(x)\rb \leq \frac{1}{M} \right\}.
\end{equation*}
\end{rem}
	
The characterized subgroups are Borel by the previous remark. Hence, if \NB $X$ is a compact metrizable abelian group and $H\in\CH(X)$, then either $|H|=\aleph_0$ or $|H|=\cc$.

\begin{rem}\label{Rzeromeasure}
According to \cite[Lemma 3.10]{CTW}, if $X$ is an infinite compact abelian group and $\vs\in \widehat{X}^\N$ contains a faithfully indexed subsequence, then $\sv(X)$ has zero Haar measure, so \NB the  index $[X:\sv(X)]$ is uncountable, as $X$ has measure 1. 
\end{rem}

It was proved in \cite{DK} and independently in \cite{BSW}, that all countable subgroups of a compact metrizable abelian group are characterized.

\begin{teo} {\rm (\cite{DK,BSW})} \label{ThDK} Let $X$ be a compact metrizable abelian group. Every countable subgroups of $X$ is characterizable.
\end{teo}

 As pointed out in \cite{VN}, the metrizability in the above theorem is a {\em necessary} condition. On the other hand, it was proved in \cite{DG} that one can reduce the study of characterized subgroups of compact abelian groups to the case of compact metrizable abelian groups in the following sense: 

\begin{teo}{\rm \cite[Theorem B]{DG}}\label{ThDG} A subgroup $H$ of a compact abelian group $X$ is characterized if and only if $H$ contains a closed $G_\delta$-subgroup $K$ of $X$ such that $H/K$ is a characterized subgroup of the compact metrizable group $X/K$. \end{teo}

 The following theorem from \cite{DG} provides information about the Borel complexity of  characterized subgroups: 

\begin{teo}\label{teoDG1} {\rm (\cite{DG})} For every infinite compact abelian group $X$ the $G_\delta$-subgroups are closed and the following inclusions hold:
\begin{equation}\label{*}
 \G(X) \subsetneqq \CH(X) \subsetneqq \FF(X)\ \  \mbox{ and } \  \  \F(X) \not \subseteq  \CH(X).
\end{equation}
If in addition $X$ has finite exponent, then $\CH(X)  \ \subsetneqq \ \F(X).$
\end{teo}

It was proved in \cite{Ga2}, that the implication in the final part of the above Theorem can be inverted:

\begin{teo} {\rm (\cite{Ga2})} \label{ThGa} $\CH(X) \subseteq  \F(X)$ for a compact abelian group $X$ if and only if $X$ has finite exponent.\end{teo}

The first named author and Gabriyelyan described completely $\CH(X)$ when $X$ has finite exponent: 

The following definition was inspired by a result from \cite{DK} establishing that certain $F_\sigma$-subgroups of compact metrizable groups
are characterized. 

\begin{defin}\label{DefCMC}{\rm (\cite{DG})}
	 A subgroup $H$ of a compact abelian group $X$ is {\em countable modulo compact} (briefly, {\em CMC}) if $H$ has a compact $G_\delta$-subgroup $K$ such that $H/K$ is countable.	
\end{defin}

We denote by $\CMC(X)$ the family of all CMC subgroups of $X$. The characterized $F_\sigma$-subgroups considered in \cite{DK}  are special 
CMC  subgroups with an additional property (namely,  $H/K$ is torsion, in the notation of Definition \ref{DefCMC}). 
By Theorem \ref{ThDK} and Theorem \ref{ThDG} (see \cite[Corollary B1]{DG}),
$\CMC(X)\subseteq  \CH(X) \cap \F(X)$ for a compact abelian group $X$. It is proved in  \cite{DG}, that $\CMC(X)=\CH(X)$ if and only if $X$ has finite exponent.

Since the group $\T$ is not of finite exponent,  Theorem \ref{ThGa} yields 
\begin{equation}\label{NEW1}
\CH(\T)\nsubseteq\F(\T). 
\end{equation}
On the other hand, the specific non-inclusion (\ref{NEW1}) was one of the main steps in the proof of the general Theorem \ref{ThGa} in \cite{Ga2}. It is curious
to note that (\ref{NEW1}) was {\em implicitly} established much earlier by Bukovsk\'y, Kholshevikova and Repick\'y in \cite{BKR}. In order to explain that we recall two notions related to sets of convergence of trigonometric series from \cite{BKR} (further information and results about sets of convergence of trigonometric series can be found in \cite{HMP,BKR}).

\begin{defin}
	\begin{itemize}
		\item A set $E\subseteq\T$ is an \emph{$N$-set} if there \NB 
		exists a sequence $\{a_n: n\in\N\}$ of non-negative real numbers  with $\sum a_n=+\infty$, such that the trigonometric series $\sum a_n\sin{2\pi nx}$ is absolutely convergent in $E$.
		\item A Borel set (resp. subgroup) $E\subseteq\T$ is a \emph{weak Dirichlet set} (resp. \emph{weak Dirichlet subgroup}) or briefly \emph{$wD$-set} (resp. \emph{$wD$-subgroup}) if for every measure $\mu$ carried by $E$, there exists an increasing sequence of positive integers $u_n$ such that $\lim_{n\to\infty}\int_E\left|e^{2\pi i u_nx}-1\right|d\mu=0$. 	\end{itemize}
\end{defin}

Let $\Nc(\T)$ denote the class of all $N$-sets of $\T$ and let $\wD(\T)$ denote the class of all $wD$-sets of $\T$.

\begin{rem}\label{rChwD}
 As an obvious consequence of the Dominated Convergence Theorem, every characterized subgroup of  $\T$  is a weak-Dirichlet subgroup \NB of $\T$.
\end{rem} 

\begin{teo}[{\cite[p.÷49]{HMP}}]\label{THMP}
  A Borel set $E\subseteq\T$ is an $N$-set if and only if it is contained in a $\sigma$-compact $wD$-set (or equivalently, \NB a $\sigma$-compact  $wD$-subgroup).
\end{teo}

As a consequence of Theorem \ref{THMP} and  Remark \ref{rChwD} one obtains the following corollary.

\begin{cor}   
  $\CH(\T)\cap\F(\T)\subseteq\wD(\T)\cap\F(\T)\subseteq\Nc(\T)\cap\F(\T).$
\end{cor}

It was proved by Arbault in \cite{Arb} that  $\A:=\tTu$, with $\us=(2^{2^n})$, is not an $N$-set. Hence, the above corollary yields $\A\notin\F(\T)$ and consequently (\ref{NEW1}). 
 Unaware of this fact, Gabriyelyan \cite{Ga2}  proved directly that the subgroup of $\T$ characterized by the sequence $u_n = n!$ ($n \in \N$) is not an $F_\sigma$-set in $\T$, so witnessing again the non-inclusion (\ref{NEW1}). Note that both sequences $(n!)$ and $(2^{2^n})$
 % used respectively in \cite{Ga2} and \cite{BKR} 
 are a-sequences. 

In the light of these two examples, it is natural to ask when $\tTu\in\F(\T)$ for a \NB {\em general sequence} $\us$. The line of the proof of Theorem \ref{ThGa} suggests to use the technique of Polishable subgroup of a Polish group. 

\subsection{About Polishability of Characterized Subgroups}\label{Polishable}

The notion of Polishable subgroup was introduced in \cite{KL}.

\begin{defin}\label{DPol}
	A subgroup $H$ of a Polish group $G$ is \emph{polishable} if $H$ satisfies one of the following equivalent conditions: 
	\begin{itemize}
		\item[(a)] there exists a Polish group topology $\tau$ on $H$ having the same Borel sets as $H$ when considered as a topological subgroup of $G$; 
		\item[(b)] there exists a continuous isomorphism from a Polish group $P$ to $H$; 
		\item[(c)] there exists a continuous surjective homomorphism from a Polish group $P$ onto $H$. 
	\end{itemize}
\end{defin}
It is a folklore fact that the topology witnessing the polishability of a subgroup $H$ is unique (see \cite{S} for an explicit proof). We denote by $\pi(H)$ this unique topology.

\begin{rem}\label{11Ott} 
Let $(X,\tau)$ be a compact metrizable abelian group and $H\le X$. It is important to mention that $\pi(H) = \tau\restriction_H$, if and only if $H$ is closed. Indeed, $\tau\restriction_H$ is Polish if and only if $H\in\G(X)$. By  \cite[Proposition 2.4]{DG} (see also Theorem \ref{teoDG1}), the class $\G(X)$ coincides with class of closed subgroups of $X$. Using the uniqueness of $\pi(H)$, we can conclude that $\pi(H) = \tau\restriction_H$ precisely when $H$ is closed. 
\end{rem}

Answering negatively a question of Kunen and \NB the first named author on whether $\F(\T)\subseteq \CH(\T)$, Bir\'o proved the   more precise theorem below. To state 
the theorem, let us recall that a non empty compact subset $K$ of an infinite compact metrizable abelian group $X$ is called a \emph{Kronecker set}, if for every continuous function $f:K\to\T$ and $\varepsilon>0$ there exists a character $v\in\widehat{X}$ such that
 \begin{equation*}
  \max\left\{\lb f(x)-v(x)\rb:x\in K\right\}<\ep.
 \end{equation*}
Bir\'o proved  that every characterized subgroup of $\T$ is Polishable and that for any uncountable Kronecker set $K$ in $\T$, the $F_\sigma$-subgroup $\langle K\rangle$ generated by $K$ is not polishable and hence not characterized:

\begin{teo}[{\cite[Theorem 2]{B}}]
	If $K$ is an uncountable Kronecker set of $\T$, then $\langle K\rangle\in\F(\T)\setminus\CH(\T)$.
\end{teo}        

A significant generalization of the above results was obtained by Gabriyelyan \cite{Ga1}: 

\begin{teo}\label{Tpolgen} {\rm \cite[Th. 1, Th. 2]{Ga1}} Let $X$ be a compact metrizable abelian group. Then $\sv(X)$ is Polishable for every sequence $\vs$ of characters of $X$.
If $K$ is an uncountable Kronecker set in $X$, then $\langle K\rangle$ is not polishable; so $\langle K\rangle$ is not characterized.
\end{teo}

The polishability of $\sv(X)$ is witnessed by a topology induced by the following metric on $X$ (see \cite{Ga1,DG} for a proof):

\begin{defin}\label{DefPolTopo}  Let $X$ be a compact metrizable abelian group with compatible metric $\delta$ and let $\vs=(v_n)$ be a sequence of characters of $X$.  For $x,y\in X$ let 
    \begin{equation}\label{def_ro}
   \varrho_\vs(x,y)=\sup_{n\in\N}\{\delta(x,y),d(v_n(x),v_n(y))\}.
  \end{equation}
\end{defin}

The  metric  topology $\tau_{\vs}$ of $X$ generated by $\varrho_\vs$ is complete (see \cite{DG}), we refer to it as {\em associated to $\vs$}. The Polish topology $\pi(\sv(X))$ witnessing the polishability of $\sv(X)$,  is the restriction of $\tau_{\vs}$ in $\sv(X)$. It was introduced for $\sv(X)$ in \cite{B,Ga1} inspired by \cite{AN} and then extended on the whole group $X$ in \cite{DG}. 

Let $\sPol(X)$ denote the collection of all Polishable subgroups of $X$.  Hence, Theorem \ref{Tpolgen} (a) can be written briefly as
\begin{equation*}
	\CH(X)\subseteq\sPol(X).
\end{equation*} 
Hjort \cite{Hj} proved that every uncountable Polish group contains polishable subgroups of arbitrary Borel complexity witnessing $\CH(X)\subsetneq\sPol(X)$ \NB in view of Remark \ref{rem1}. In \cite{Ga4,Ga1} Gabriyelyan proved that there exists a compact group $X$ witnessing $\CH(X)\cap\F(X)\subsetneq\sPol(X)\cap\F(X)$. More precisely, he proved that there exists a polishable $F_\sigma$ subgroup of $\T^\N$ that is not characterized.

%%%%%%%%%%%%%%%%%%%%%%%%%%%%%%%%%%%%%%%%%%%%%%%%%%%%%%%%%%%%%%%%%%%%%%%%%%%%%%%
% ----------------------------  M a i n  R e s u l t s   ----------------------
%%%%%%%%%%%%%%%%%%%%%%%%%%%%%%%%%%%%%%%%%%%%%%%%%%%%%%%%%%%%%%%%%%%%%%%%%%%%%%%

\section{Main Results}

In this paper we provide a general tool to study the special aspect of Problem \ref{probBor} concerning $F_\sigma$-subgroups (Theorem A) with particular emphasis on the case $X=\T$. It is based on a complete description, in terms of an appropriate topology defined starting from a sequence which characterizes a subgroup of $X$. 
We provide a general criterion in Theorem A to determine whether $\tTu$ is an $F_\sigma$-set.
In the case of $X = \T$ and  a-sequences one can prove much more than the mere non-inclusion $\CH(\T)  \not \subseteq  \F(\T)$ established in Theorem \ref{teoDG1}. Namely, \emph{the only proper $F_\sigma$-subgroups of $\T$, characterized by an a-sequence are the countable ones}. \NB 
Moreover, for a-sequences we will find an explicit subgroup that witnesses the separability of $(\tTu,\pi(\tTu))$, namely the torsion subgroup $t(\tTu)$. 

In the general case of a compact metrizable abelian group $X$ we have defined above a second assignment
with domain $\widehat{X}^\N$, namely $\vs \mapsto \tau_{\vs}$. This assignment defines a second equivalence relations in $\widehat{X}^\N$ by $\us\approx\vs$ if and only if $\tau_{\us}=\tau_{\vs}$.  We discuss the connection between the equivalence relations $\sim$ and $\approx$. In general, $\us\approx\vs \not \Rightarrow \us\sim\vs$ (see Remark \ref{RemXXX}). 
On the other hand, the topology $\pi(\sv(X))$, witnessing the polishability of $\sv(X)$ is unique. It coincides with the restriction of $\tau_{\vs}$ to $\sv(X)$, so this restriction depends only on the characterized subgroup and not on the sequence that characterizes it. In \S \ref{Stest} we prove that if the characterized subgroup $\sv(X)$ is an $F_\sigma$-subgroup, then the uniqueness of the topology $\tau_{\vs}$ extends to the whole group in the sense that  $\sv(X)=\su(X)$ implies $\tau_{\vs}=\tau_{\us}$ (i.e., $\us\sim\vs\Rightarrow\us\approx\vs$) in case $\sv(X) = \su(X)$ is an $F_\sigma$-subgroup, see Corollary A1.  We conjecture that this implication fails in the general case: 

\begin{que}\label{QtildeDoppiaTilde} Does $\us\sim\vs\Rightarrow\us\approx\vs$ in the general case?
\end{que}

% \begin{rem}
%  Clearly if $X$ has finite exponent, then $\LL(X)=0$. The converse is not true, for example $X=\T$. Hence in that case $\CH(\T)\setminus\sLC(\T)\neq\emptyset$. As we will see $\sLC(\T)=\{\text{countable subgroups of }\T\}\cup\{\T\}$ and hence all uncountable characterized subgroups (e.g. $\tTu$ for $\us=(n!)$) witness the previous inequality.
% \end{rem}

\subsection{Results for compact metrizable abelian groups}

For a compact metrizable  abelian group $(X,\tau)$, we introduce another group topology $\tau_{\vs}^*$ with $ \tau\subseteq\tau_{\vs}^*\subseteq\tau_{\vs}$ in \S\ref{Stest}
 (see Definition \ref{DefTestTopo}). 
 We refer to $\tau_{\vs}^*$ as the \emph{$F_\sigma$-test topology} with respect to the sequence $\vs$. The motivation for such a choice is clear from the next theorem: 	

\begin{nteo}{Theorem A}
  Let $X$ be a compact metrizable abelian group $(X,\tau)$ and $\vs\in \widehat{X}^\N$. Then 
  \begin{equation*}
   \sv(X)\in\F(X)\Longleftrightarrow\sv(X)\in\tau_{\vs}^*.
  \end{equation*}
\end{nteo}

As a corollary of Theorem A, one can prove that the topology $\tau_\vs$ does not depend on the choice of the characterizing sequence $\vs$, whenever the characterized subgroup $\sv(X)$ is an $F_\sigma$-subgroup.

\begin{nteo}{Corollary A1}
 Let $X$ be compact metrizable abelian group and $\vs$ be a sequence of characters with $\sv(X)\in\F(X)$. If $\us$ is a sequence of characters with $\us\sim\vs$, then $\us\approx\vs$.
%  \item[(b)] If $X$ is also infinite, then the following are equivalent:  
%  \begin{itemize}
%      \item[(b$_1$)] $(X,\tau_{\varrho_\vs})$ is Polish; 
%      \item[(b$_2$)] $(X,\tau_{\varrho_\vs})$ is separable;
%      \item[(b$_3$)] $\sv(X)$ is a closed finite-index (so, open) subgroup of $X$.
%      \item[(b$_4$)] there exists a sequence $\us  $ that splits in finitely many constant sequences $\us^{(i)}=(u_n^{(i)})$, $i = 1,2,\ldots,m$
%      such that each $\us^{(i)}$ is determined by some torsion character $u_n^{(i)}$ of $X$.  
\end{nteo}

The next corollary of Theorem A shows that $(X,\tau_{\vs})$ is Polish only in some very special cases.

\begin{nteo}{Corollary A2} Let $(X,\tau)$ be a compact metrizable abelian group and $\vs\in\widehat{X}^\N$. Then the following are equivalent: 
\begin{itemize}
   \item[(a)] $(X,\tau_{\vs})$  is Polish; 
   \item[(b)] $\sv(X)$ is a $\tau$-open subgroup;
   % of finite index; 
   \item[(c)] $\tau_\us = \tau$;  
   \item[(d)] $\vs$ has no faithfully indexed subsequences and every $v_n$ that appears infinitely many times in $\vs$ is a torsion character. 
\end{itemize}
\end{nteo}

\begin{rem}\label{RnoFaith} 
The second part of in item (d) of Corollary A2 cannot be omitted. Indeed it may occur that $\vs$ has no faithfully indexed subsequences, but $\mu(\sv(X))=0$ and hence it has uncountable index. Take for example $X=\T$ and $\vs=(2)$ the constant sequence with term 2, then $\sv(\T)=\{0,\frac{1}{2}\}$ is a finite closed subgroup of measure $0$ and of uncountable index.
\end{rem}

The next theorem, proved in \cite{DI_}, describes the cases when $\tau_{\us}$ and its restriction $\pi(\su(X)) = \tau_{\us}\restriction_{\su(X)}$ are discrete. Since the latter is  a finer Polish topology on $\su(X)$, the group $(\su(X), \pi(\su(X)))$ is discrete if and only if $\su(X)$ is countable. Moreover, $\sv(X)$ is countable if and only if $\tau_{\vs}=\tau_{\vs}^*$ is discrete: 

\begin{teo}{}\label{ex_B}{\rm (\cite{DI_})}
	 Let $X$ be a metrizable compact abelian group and $\vs\in \widehat{X}^\N$. Then the following are equivalent\NB: 
\begin{itemize}
   \item[(a)] $\sv(X)$ is countable;
   \item[(b)] $|\sv(X)|<\cc$;
   \item[(c)] $\tau_{\vs}$ is discrete;
   \item[(d)] $\tau_{\vs}^*$ is discrete;
   \item[(e)] $\tau_{\vs}^*\restriction_{\sv(X)}$ is discrete;
   \item[(f)] $\pi({\sv(X)})$ is discrete.
\end{itemize}
\end{teo}

\begin{rem}\label{RemXXX}
 The above theorem yields that in general $\us\approx\vs$ does not imply \NB  $\us\sim\vs$. Indeed, all countable subgroups are characterized (by Theorem \ref{ThDK}) and 
 have the same associated topology, namely the discrete one (by Theorem \ref{ex_B}). 
\end{rem}

% Following \cite{DG}, call a subgroup $H$ of a compact abelian group $X$ {\em countable modulo compact} (briefly, {\em CMC}) if $H$ has a compact $G_\delta$-subgroup $K$ such that $H/K$ is countable.
% Clearly, a CMC subgroups are  characterized subgroup and $F_\sigma$ (as we can see in the next proposition). It is proved in  \cite{DG}, that every characterized subgroup $H$ of $X$ is CMC if and only if $X$ has finite exponent.
% As we revealed in advanced, the following Theorem refines the description of $\sLC(X)$.

\subsection{Results for the circle group}

The  subgroups of $\CH(\T)$ are characterized by sequences $\us$ of non-zero integers. A useful feature of such a sequence is  the sequence of ratios $(q_n)=(\frac{u_n}{u_{n-1}})$. Furthermore, 
one can define  
$$
q_\us: = \limsup_n {q_n}, \  \ Q_\us: = \sup_n q_{n}, \ \ S_\us= \{m\in \N: q_m = Q_\us\}  \ \mbox{  and  } \ S_\us ^*= \{m_k\in \N: q_{m_k} = q_\us\}.
$$
Clearly,   $q_\us \leq Q_\us$ and $S_\us ^*$ is infinite when $Q_\us< \infty$. Moreover, $q_\us = Q_\us$ precisely when $S_\us$ is infinite (and this case $S_\us = S_\us ^*$). The sequence of ratios $(q_n)$ is bounded precisely when  $q_\us< \infty$, this occurs if and only if   $Q_\us< \infty$.  

By Theorem \ref{TEgg}, the subgroup $\tTu$ is countable, when $Q_\us < \infty$. 
Hence, the topology $\tau_{\us}$ restricted to $\tTu$ is discrete, and actually  $(\T, \tau_{\us})$ is discrete, by Theorem \ref{ex_B}. In addition we can specify which balls are trivial.

\begin{nteo}{Proposition B} If $\us$ is a strictly increasing sequence of positive integers with $Q_\us < \infty$, 
then $\tau_{\us}$ is discrete. More precisely, $B_{\frac{1}{2 Q_\us}}^{\varrho_\us}(0) = \{0\}$ in $\T$. 
\end{nteo}

Now we investigate in more detail the case of an a-sequence $\us$ with $q_\us < \infty$. Since the ratios $q_n$ are integers, also $q_\us$ and $Q_\us$ are integers. 
 We shall study the ball $B_{\frac{1}{q_\us }}^{\us}(0)$. 
% when  $\us$ is an a-sequence with 
%There exists a co-finite subsequence $\vs$ of $\us$, such that $Q^\vs = q^\vs = q_\us$. 
%one can assume that the set $S_C = \{m\in \N: q_m = C\}$ is infinite. 
%Indeed, $\varrho_\vs \leq \varrho_\us$, so $B_{\varepsilon}^{\varrho_\us}(0) \subseteq B_{\varepsilon}^{\varrho_\vs}(0) $ for every $\ep >0$. 
%So if $B_{\varepsilon}^{\varrho_\vs}(0) $ is finite (trivial) for some ${\varepsilon}> 0$, then the same will hold for $B_{\varepsilon}^{\varrho_\us}(0) $. 
%On the other hand, we shall see that if  $|B_{1/C}^{\varrho_\us}(0)|=\mathfrak c$, then the same will hold for $B_{1/C}^{\varrho_\us}(0) $. 
%
It turns out that the closed $\varrho_\us$-ball $\{x\in \T: \varrho_\us(x,0) \leq 1/q_\us\}$ centered at 0 has size $\mathfrak c$. Moreover, if $S_\us^*$ is a big set in $\N$ (i.e., the gaps between two consecutive elements of $S_\us^*$ are bounded),  then actually $|B_{\frac{1}{q_\us}}^{\us}(0)| = \mathfrak c$.

\begin{nteo}{Proposition C}
Let $\us$ be an a-sequence with bounded ratio sequence and let $\delta_k:= m_{k+1}-m_k$, where $S_\us ^*= \{m_k\in \N: q_{m_k} = q_\us\}$. 
\begin{itemize}
\item[(a)] there are $\mathfrak c$ many elements $x\in \T$ with $\varrho_\us(x,0) = \frac{1}{q_\us}$ when the sequence $\delta_k$ is unbounded; 
\item[(b)]  if the sequence $\delta_k$
 %of differences $m_{k+1}-m_k$ in $S_\us ^*= \{m_k\in \N: q_{m_k} = q_\us\}$
  is bounded, then $\left|B_{\frac{1}{q_\us}}^{\varrho_\us}(0)\right|=\cc$.
\end{itemize}
%  the closed $\varrho_\us$-ball $\{x\in \T: \varrho_\us(x,0) \leq 1/q_\us\}$ centered at 0 has size $\mathfrak c$; 
 \end{nteo}

Taking an a-sequence with $q_\us = Q_\us$ and $S_\us = S_\us^*$ big in  Proposition C will show that the radius $\frac{1}{2Q^\us }$ in Proposition B cannot be increased to $\frac{1}{Q^\us}$.

The following question arises naturally from the two previous propositions.

\begin{que}
 Let $\us\in\Z^\N$ have bounded sequence of ratios. What can be said about the size of $B_{\varepsilon}^{\varrho_\us}(0)$, whenever $\frac{1}{2Q^\us}<\varepsilon<\frac{1}{q_\us}$?
\end{que}

In \S\ref{Stors}, as we previously announced, for a-sequences in $\T$, we will find a specific subgroup of $\tTu$ witnessing the separability of $(\tTu,\tau_{\us})$. This subgroup is $t(\tTu)$, i.e., the torsion subgroup of $\tTu$. This need not be true if $\us$ is not an a-sequence as the next remark shows. \NB

\begin{exa}\label{Rem_} If $\us$ is not an a-sequence, then $t(\tTu)$ need not be dense in $(\tTu,\tau_\us)$. \NB 
This may occur if the ratio sequence $(q_n)$ is divergent (so $\tTu$ is uncountable) and $t(\tTu)$
is finite (hence non-dense respect to any Hausdroff group topology on $\T$,  so non-dense in $(\T,\tau_\us)$ either).  
An example to this effect is the sequence of prime numbers $\us$. %In order to find such a sequence $\us$, it suffices to take each $u_n$ to be prime with $u_n > nu_{n-1}$. 
 Then $t(\tTu)$ is trivial, by Proposition \ref{pTorStruc}. 
  %Let $\us\in\PP^\N$ such that $u_n\ge nu_{n-1}$, then $t(\tTu)$ is not dense in $\tTu$.
 %\item[(b)] Still by Proposition \ref{pTorStruc}, every subgroup characterized by a subsequence of the Fibonacci sequence such that $q_n$ is divergent, has a non-dense torsion subgroup.
%\end{itemize}
\end{exa}

\begin{nteo}{Proposition D}
  If $\us$ is an a-sequence, then $t(\tTu)$ is a dense subgroup of $(\tTu, \tau_{\us}\restriction _{\tTu})$. 
\end{nteo}

\begin{nteo}{Theorem E} 
 
The following are equivalent for an a-sequence $\us$ in $\Z$:
\begin{itemize}
    \item[(a)] $(q_n)$ is bounded;
    \item[(b)] $\tTu\le\Q/\Z$;
    \item[(c)] $\tTu$ is countable;
    \item[(d)] $\tTu\in \F(\T)$;
    \item[(e)] $\tTu$ is $\tau_\us^*$-open;
    \item[(f)] $\tTu$ is $\tau_\us$-open.
\end{itemize}
\end{nteo}
%In case $(b_n)$ is splitting, they imply \begin{itemize}
%\item[(e)]  $t_{\underline{m}}(\T)$ is not  divisible. \end{itemize}
%In case $\Q/\Z\subseteq t_{\un m}(\T)$ all five conditions are equivalent. \end{teo}

\begin{ncor}{Corollary E} If $\us$ is an a-sequence with unbounded sequence of ratios, then $\tTu\notin\F(\T)$.
\end{ncor}

This corollary covers the results from \cite{BKR} and  \cite{Ga2} with $\us=(2^{2^n})$ and $\us=(n!)$.

Some implications of Theorem E hold for a general sequence of integers $\us$, while others are no more valid. For more details see  \cite{DI_}, where 
the main results of this paper (as well as  some related results from \cite{DH}) were announced without a proof.

\section{Proofs}\label{SPol}

\begin{notation} Let $\us=(u_n)$ be a sequence of characters of $X$, $\varrho_\us$ be as in Definition \ref{DefPolTopo}, $\|\;\|_{\varrho_\us}$ the norm induced by $\varrho_\us$, $B_\ep^{\varrho_\us}(0)$ be the $\varrho_\us$-ball of radius $\ep$ around 0 in $X$, $\tau_{\us}$ be the topology induced by $\varrho_\us$ in $X$ and  $\tau_{\us}\restriction\tTu$ its restriction on $\tTu$. \end{notation}

% ----------------------------------------   T e s t   T o p o l o g y -----------------------------------------

\subsection{The general case of compact metrizable abelian groups}\label{Stest}

Now we define the $F_\sigma$-test topology $\tau_{{\vs}}^*$.
% coarser  than $\tau_{{\vs}}$.
	 
	\begin{defin}\label{DefTestTopo} For a sequence  $\vs$ of characters of a compact metrizable abelian group $(X,\tau)$,  $\tau_{\vs}^*$ is the group topology on $X$ with
 filter of neighborhoods of $0$ in $X$ is generated by the family
	\begin{equation}\label{***}
	\left\{W_n=\overline{B^{{\varrho_\vs}}_{1/n}(0)}:n\in\N\right\}. 
	\end{equation}
\end{defin}
	\begin{rem}\label{TestTopo}
 Let us see that $\tau_{\vs}^*$ is a metrizable group topology with $ \tau\subseteq\tau_\vs^*\subseteq\tau_\vs$. To prove that $\tau_\vs^*$ is a metrizable group topology, it suffices to note that (\ref{***}) is a countable decreasing chain of symmetric sets, such that 
$$
W_{2n} + W_{2n} = \overline{B^{{\varrho_\vs}}_{1/2n}(0)} +\overline{B^{{\varrho_\vs}}_{1/2n}(0)} 	\subseteq  \overline{ B^{{\varrho_\vs}}_{1/2n}(0)  + B^{{\varrho_\vs}}_{1/2n}(0)}   \subseteq \overline{ B^{{\varrho_\vs}}_{1/n}(0) }   = W_n.	
$$
	
	For the inclusion $\tau_\vs^*\subseteq\tau_\vs$ it suffices to note that for every $W_n$  there exists $m>n$ such that $B_{1/m}^{\varrho_\vs}(0)\subseteq W_n$. 
	
	To prove the inclusion $\tau\subseteq\tau_\vs^*$ take $0 \in U\in\tau$. Then $U$ contains a closed ball $\overline{B_{1/n}^\delta(0)}$ in $\tau$  for some $n\in\N$,  
	where $\delta$ is a compatible metric with $\tau$. Furthermore, $ B_{1/n}^{\varrho_\vs}(0)\subseteq B_{1/n}^\delta(0)$, as $\varrho_\vs(x,y)
	\ge\delta(x,y)$ for every $x,y\in X$, according to (\ref{def_ro}). Hence, $W_n \subseteq \overline{B_{1/n}^\delta(0)}\subseteq U$, so $U \in \tau_\vs^*$. 
	
	%For the metrizability of $\tau_{\varrho}^*$, it suffices to note that by construction $\tau_{\varrho}^*$ has a countable base of neighbourhood in $0$.
	%
%	To prove that $\tau_\varrho^*$ is a group topology, one can proceed in this way. Let $f(x,y)=xy^{-1}:X\times X\to X$ and $A=f^{-1}(W_n)$ for a certain $n\in\N$. By the previous inclusions, $A$ is an open subset in the $\tau_\varrho$-product topology  $\tau_\varrho\times\tau_\varrho$. Hence, for all $(x,y)\in A$ there exist two $\tau_\varrho$-balls $B_x$ and $B_y$ such that $(x,y)\in B_x\times B_y\subseteq A$. Therefore $A=\bigcup_{(x,y)\in A}B_x\times B_y$. Since $A$ is obviously also $\tau$-closed, the following holds \[A=\overline{\bigcup B_x\times B_y}\supseteq\bigcup \overline{B_x}\times\overline{B_y}\supseteq \bigcup B_x\times B_y=A.\]
%	This proves that $A$ is a $(\tau_\varrho^*\times\tau_\varrho^*)$-open.	
\end{rem}		
% \begin{teo}\label{TfurTop}
%   Let $\vs$ a sequence of characters of a compact abelian group $(X,\tau)$. Then $\sv(X)\in\SF(X)$ if and only if $\sv(X)\in\tau_{\varrho_\us}^*$.
% \end{teo}
\medskip
\noindent{\bf Proof of Theorem A.}
  Suppose that $ \sv(X)= \bigcup_m F_m$ for some closed sets $F_m$ of $(X,\tau)$. As $ \sv(X)$ is a subgroup, we can assume without loss of generality that $F_m = -F_m$ is symmetric for  each $m\in \N$. Obviously, these $F_m$ are ${\tau_\vs}$-closed as well. So, applying the Baire category theorem to the Polish space $(\sv(X),{\tau_\vs})$ we deduce that  $F_{m_0}$ has a non-empty interior for some $m_0$. Since  $F_{m_0}$ is closed in $(X,\tau)$, there exists some $x_0 \in F_{m_0}$ and  $n\in\N$ so that 
\begin{equation}\label{eq1}
\overline{B_{1/n}^{\varrho_\vs} (x_0) }= x_0 +\overline{B_{1/n}^{\varrho_\vs} (0) } \subseteq F_{m_0}\subseteq \sv(X).
\end{equation}
Hence $\sv(X)$ is $\tau_\vs^*$-open, since it is a subgroup.

Converse, let $\sv(X)\in\tau_\vs^*$. Thence there exists $n\in\N$ such that $W_n\subseteq\sv(X)$. Let $D$ be a $\tau_\vs$-dense countable subset of $\sv(X)$. Then $D$ is also $\tau_\vs^*$-dense by Remark \ref{TestTopo}, so $D+W_n=\sv(X)$. Hence, $\sv(X)=\bigcup_{d\in D}d+W_n$. Each $d+W_n$ is $\tau$-closed, as $W_n$ are $\tau$-closed.  Therefore, $\sv(X)$ is an $F_\sigma$-set. 
\hfill$\Box$\medskip

\medskip\noindent{\bf Proof of Corollary A1.}
 Let $H:=\sv(X)=\su(X)$. Recall that $\tau_{\vs}$ and $\tau_{\us}$ coincide when restricted to $H$, since the topology witnessing the polishability of $H$ is unique. Moreover, as $H\in\F(X)$, $H$ is $\tau_{\vs}^*$-open  and hence also $\tau_{\vs}$-open. 
Analogously, $H$ is $\tau_{\us}$-open. Hence  $\tau_{\vs}$ and $\tau_{\us}$ coincide on a subgroup 
that is  open in both topologies. Therefore, they coincide in $X$. 
\hfill$\Box$\medskip

\medskip\noindent{\bf Proof of Corollary A2.}
(a) $\leftrightarrow $ (b)  The group $(X,\tau)$ is Polish and $\tau_\vs\geq \tau$. Therefore, $(X,\tau_\vs)$ is Polish if and only if $\tau=\tau_\vs$. 

(b) $\to $ (c) Assume that $\tau=\tau_\vs$ (i.e., $(X,\tau_\vs)$ is Polish). Then $H=\sv(X)$, is Polish in $\tau\restriction H = \tau_\vs\restriction H$. 
By Remark \ref{11Ott}, 
%This happens if and only if
 $H$ is $\tau$-closed, in particular $H\in\F(X,\tau)$. Theorem A implies $H\in\tau_{\vs}=\tau$, i.e., $H$ is also $\tau$-open. 
 
(c) $\to $ (d) First we note that the quotient $X/H$ is compact and discrete, and hence  finite. So $H$ is a finite index subgroup of $X$.  

Let $\mu$ be the normalized Haar measure in $X$. Since the $\tau$-open subgroup $H=\sv(X)$ has a finite index, one has $\mu(H)>0$. 
By Remark \ref{Rzeromeasure} $\vs$ has no faithfully indexed subsequences. If $v_n$ appears infinitely many times in $\vs$, then $v_n$ vanishes on $H$. Hence
$v_n$ factorizes through the quotient map $X \to X/H$. Since $X/H$ is finite, the image  $v_n(X)$ of $v_n$ is a finite subgroup of $\T$. If $m$ is the exponent of $v_n(X)$, 
then $mv_n= 0$, i.e., $v_n$ is torsion. 

(d) $\to$ (b) If $\vs=(v_n)$ has no faithfully indexed subsequences, then there exist finitely many characters $\left\{ v_{n_k}\mid0\le k \le m\right\}$
 that occur infinitely many times in the sequence $\vs$. Then each character $ v_{n_k}$ is torsion and vanishes on $H$, so $\sv(H)=\bigcap_{k=0}^m\ker v_{n_k}$ is $\tau$-closed.
Since each  $v_{n_k}$ is torsion, the subgroup $ v_{n_k}(X)$ of $\T$ is of finite exponent, hence it is finite. This proves that 
$\ker v_{n_k}$ is a finite index subgroup of $X$. Therefore, the $\tau$-closed subgroup $H$ has a finite index, so $H$ is $\tau$-open. 
This proves that $H=\sv(X)$ is also $\tau_\vs$-open. Moreover, being $H$-closed, $\tau\restriction H=\pi(H)=\tau_\vs\restriction H$ by Remark \ref{11Ott}. 
Therefore, $\tau$ and $\tau_\vs$ coincide on the subgroup $H$ which is open in both topologies. Hence, $\tau=\tau_\vs$. \hfill$\Box$\medskip

\subsection{Description of $\tTu$ for an a-sequence $\us$}\label{secaseq}

Throughout this section $\us=(u_n)$ will denote an a-sequence.

The following fact, providing a canonical representation of $x\in [0,1)$, is well-known (see for example \cite{Pe}). 

\begin{defin}[Canonical Representation with respect to $\us$]\label{Def_canonical}
		Let $x\in [0,1)$. Then exists a unique sequence $(c_n)\in\Z^{\N}$ such that $0\leq c_n<q_n$ for every $n$,
	\begin{equation}\label{canonical}
		x= \sum_{n=1}^\infty\frac{c_n}{u_n}, 
	\end{equation}
and $c_n<q_n-1$ for infinitely many $n$. %Call this representation \emph{canonical representation}. 
	\end{defin}

\begin{rem} \label{NewRem} If (\ref{canonical}) is a canonical representation, then 
$$
\left(\frac{c_{n{}+1}}{q_{n{}+1}}+\cdots+\frac{c_{n+t}}{q_{n+1}\cdots q_{n+t}}+\cdots\right)<1
$$
for all $t,n \in \N$. 
\end{rem}

\begin{notation}
For $x\in\T$ with canonical representation (\ref{canonical}), let
\begin{itemize}
  \item $\SUX=\{n\in\N\mid c_n\neq0\}$ and 
  %\item $\SUQX=\{n\in\N\mid c_n=q_n-1\}$.
\end{itemize}
	
Call an infinite set $A$ of naturals 
\begin{itemize}
  \item \NB {\em $\us$-bounded} if the sequence $\{q_n: n\in A\}$ is bounded;
  \item \NB {\em $\us$-divergent} if the sequence $\{q_n: n\in A\}$ diverges to infinity.
\end{itemize}
\end{notation}

The following results are corollaries of Theorem \cite[Theorem 2.3]{DI} covering the particular cases used here. That theorem in full generality
is not used here.

\begin{teo}\label{bounded}{\rm \cite[Corollary 2.4]{DdS}} Let $x\in\T$. If $\SUX$ is \NB $\us$-bounded, then the following are equivalent: 
\begin{itemize}
\item[(i)] $x\in \tTu$;
\item[(ii)] $c_n=0$ for almost all $n\in \N$.
\end{itemize}
\end{teo}	
	
\begin{teo}[\cite{DI} Corollary 3.4]\label{unbounded} Suppose that $x\in\T$ has \NB $\us$-divergent support. Then  $x\in \tTu$ if and only if the following two conditions are satisfied: 
\begin{itemize}
    \item[(1)] \NB $\lim_{n\in\SUX}\frac{c_{n}}{q_{n}}=0$ in $\T$ and  
    \item[(2)] $\lim_{n\in I'}\frac{c_{n}}{q_{n}}=0$ in $\R$ for every infinite \NB $I'\subseteq \SUX$ such that $I'-1$ is \NB $\us$-bounded. 
\end{itemize}
\end{teo}

%\begin{proof}
%Let $I = \SUX$. If $\varphi(x)\in\tTu$, then (1) holds true by item (b) of Theorem \ref{teo1} applied to $A=I$. Assume that $A=I'-1$ is $q$-bounded for some infinite $I'\subseteq I$. 
%Then $A\cap I$ is finite, as $I$ is $q$-divergent. Then by $(a_2)$ applied to $A$, $\lim_{n\in I'}\frac{c_{n}}{q_{n}}= \lim_{n\in A}\frac{c_{n+1}}{q_{n+1}}=0$ in $\R$. This proves  (2) and the necessity. 
%
%To establish the sufficiency, assume that (1) and (2) hold true. According to Theorem \ref{teo1}, to prove that $\varphi(x)\in\tTu$ we have to check (a) and (b). Since (b) immediately follows from (1), we a%\subsubsection{The Main Theorem}re left with (a). 
%Let $A$ be an infinite $q$-bounded set in $\N$. Then $A\cap I$ is finite, so we need to check only ($a_2$), i.e. $ \lim_{n\in A}\frac{c_{n+1}}{q_{n+1}}=0$ (since the final assertion of ($a_2$) follows from this equality, as mentioned in the final part of the proof of the necessity of ($a_2$)). Let $I' = (A+1) \cap I$. If this set is infinite, then (2) applies and we are done. If $I'$ is finite, we conclude that $c_n=0$ for almost all $n\in A+1$ and hence $\lim_{n\in A+ 1 }\frac{c_{n}}{q_{n}} = 0$.%$A \setminus I$ is infinite and therefore $\lim_{n\in A}\frac{c_{n+1}}{b_{n+1}}= \lim_{n\in (A+ 1) \setminus I }\frac{c_{n}}{b_{n}} = 0$, as $c_n=0$ for all $n (A+ 1) \setminus I $. 
%\end{proof}

The following result was established in \cite{DdS}: 

\begin{cor}\label{unbounded*} Suppose  $x\in\T$ has \NB $\us$-divergent support. Then  $x\in\tTu$  whenever $\lim_{n\in \SUX}\frac{c_{n}}{q_{n}}=0$ in $\R$.
\end{cor}

The next Corollary,  following  obviously from the previous one, will be useful in the applications. 

\begin{cor}\label{unbounded**} Suppose  $x\in\T$ has \NB $\us$-divergent support. Then  $x\not\in\tTu$  whenever $\frac{c_{n}}{q_{n}}$ does not converge to 0 in $\T$.
\end{cor}

\subsection{Proofs for the circle group}

\subsubsection{The $\varrho_\us$-balls in $\T$}

 Now we recall some fact on the norm of the \NB torus and then we prove Proposition C. 

 \begin{rem}\label{rBalls} Since $\diam(\T)=\frac{1}{2}$, $\T=B_\ep^\varrho(0)\subseteq B_\ep^d(0)$ whenever $\ep\ge\frac{1}{2}$. 
\end{rem}
 
\begin{rem} \NB As already mentioned above, we identify elements of $\T$ with their \NB unique preimage in $[0,1)$ and use the inequality $\lb nx\rb\le nx$ (in $\R$) for $x\in\T$ and $n\in\N$.  If $nx\le \frac{1}{2}$, then $\lb nx\rb= nx$. 
\end{rem}

\medskip\noindent{\bf Proof of Proposition B.} Let $\ep=\frac{1}{2Q^\us}$ and
\begin{equation}\label{eq*}
x\in B_\ep^{\varrho_\us}(0). 
\end{equation}
Our aim is to prove that $x=0$. From (\ref{eq*}) we deduce that 
\begin{equation}\label{eq**}
\lb x\rb<\ep \ \mbox{ and }\  \lb u_nx\rb<\ep \ \mbox{ for all }\ n\in\N. 
\end{equation}

\noindent {\bf Claim 1.} {\em If $x<\ep$, then $\lb u_nx\rb=u_nx<\ep$ for all $n\in\N$.}

\smallskip

To prove the claim we argue by induction. For $n=0$, $u_0\leq Q^\us$, so  $u_0x\le Q^\us x<Q^\us \ep<\frac{1}{2}$. Hence, $\lb u_0x\rb=u_0x<\ep$. 
Assume that $n\ge0$ and $u_nx<\ep$ holds true. Then $u_{n+1}x\le Q^\us u_nx<Q^\us \ep<\frac{1}{2}$ and hence $u_{n+1}x<\frac{1}{2}$. Therefore, (\ref{eq*}) yields $\ep>\lb u_{n+1}x\rb=u_{n+1}x$. This concludes the proof of the claim.
%If $x<\ep$, by induction $u_nx<\ep$ for all $n\in\N$.
  
  \smallskip   
  \smallskip 

According to (\ref{eq**}), either  $x<\ep$, or $1-\ep < x<1$. In case $x<\ep$, Claim 1 yields $u_nx<\ep$ and hence $x<\frac{\ep}{u_n}$ for all $n\in\N$. Since $\us$ is strictly increasing, this entails $x=0$. 

 Let us see now that the case $1-\ep < x<1$ cannot occur. Indeed, if $1-\ep < x<1$, then $y=1-x>0$ satisfies $y\in B_\ep^{\varrho_\us}(0)$ and $0<y<\ep$. Hence, $y=0$ by the previous argument,  a contradiction. 
 \hfill$\Box$\medskip

  In order to demonstrate Proposition C we need the following lemma.
  
  \begin{lemma}\label{Lstima}
   Let $\us$ be a strictly increasing a-sequence with $q_\us < \infty$, let
 \begin{equation}\label{eqStype}
 S=\{n_1<n_2\cdots<n_k<\cdots\}\subset S_\us^* \text{ with } d_i = n_{i+1}-n_i\ge2 \mbox{ for } 1= 1,2, \ldots
 \end{equation}
and let 
 \begin{equation}\label{eqxsgen1} 
   x_S=\frac{1}{u_{n_1}}-\frac{1}{u_{n_2}}+\frac{1}{u_{n_3}}-\frac{1}{u_{n_4}}+\cdots . 
  \end{equation}
\begin{itemize}
\item[(a)]    If $n_k\le n<n_{k+1}$, then $\lb u_{n-1}x_S\rb\le\lb u_{n_k-1}x_S\rb$ and 
   \begin{equation}\label{NEWeq}
\frac{1}{ q_\us} - \frac{1}{2^{d_{k} -1} q_\us^2 }\le  \frac{1}{q_{n_k}}-\frac{1}{q_{n_k}\cdots q_{n_{k+1}}}\le\lb u_{n_k-1}x_S\rb \le
\frac{1}{q_{n_k}}-\frac{1}{q_{n_k}\cdots q_{n_{k+1}}}+\frac{1}{q_{n_k}\cdots q_{n_{k+2}}}.
   \end{equation}
\item[(b)]  $\varrho_\us(x_S,0) = \max\{\|x_S\|, \sup\{\lb u_{n_k-1}x_S\rb: k \in \N \} \}$. 
\item[(c)] If $S' \ne S$ is as in (\ref{eqStype}), then $x_S\ne x_{S'}$. 
\end{itemize}
\end{lemma}

\begin{proof} 
   (a) Clearly, one has 
% \begin{equation}\label{eqNew1} 
%\frac{1}{q_{n_k}}-\frac{1}{q_{n_k}\cdots q_{n_{k+1}}} \le   
\ $ \lb u_{n_k-1}x_S\rb=\lb \frac{1}{q_{n_k}}-\frac{1}{q_{n_k}\cdots q_{n_{k+1}}}+\frac{1}{q_{n_k}\cdots q_{n_{k+2}}}-\cdots\rb$, 
  %\end{equation}
so 
   \begin{equation}\label{eqNew1.5} 
  \frac{1}{q_{n_k}}-\frac{1}{q_{n_k}\cdots q_{n_{k+1}}} \le   \lb u_{n_k-1}x_S\rb \le  \frac{1}{q_{n_k}}-\frac{1}{q_{n_k}\cdots q_{n_{k+1}}}+\frac{1}{q_{n_k}\cdots q_{n_{k+2}}} 
.
  \end{equation}
  
This proves (\ref{NEWeq}). 

To prove the first assetrtion we need to see that for all values of $n$ satisfying $n_k\le n<n_{k+1}$, the biggest value of $\|u_{n-1}x_S\|$ is obtained for $n=n_k$. Indeed, as $n < n_{k+1}$, one has 
 \begin{equation}\label{eqNew2} 
    \lb u_{n-1}x_S\rb=\lb \frac{1}{q_{n}\cdots q_{n_{k+1}}} -\frac{1}{q_{n}\cdots q_{n_{k+2}}}+\frac{1}{q_{n}\cdots q_{n_{k+3}}}-\cdots\rb \le \frac{1}{q_{n}\cdots q_{n_{k+1}}} 
   \le \frac{1}{2^{n_{k+1} - n} q_\us} \le \frac{1}{2q_\us}.
 \end{equation}
Obviously, $\frac{1}{2q_\us} \le \frac{1}{ q_\us} - \frac{1}{2^{d_{k} -1} q_\us^2 }$. Hence, (\ref{NEWeq})  and (\ref{eqNew2})
yield $\lb u_{n-1}x_S\rb\le\lb u_{n_k-1}x_S\rb$. 

\medskip

(b) Let us first prove that if $0<n < n_{1}$, then $\lb u_{n-1}x_S\rb\le\lb u_{n_1-1}x_S\rb$.  Indeed,
 \begin{equation}\label{eqNew3} 
   \lb u_{n-1}x_S\rb
   \le \frac{1}{2^{n_{1} - n} q_\us} \le \frac{1}{2q_\us} \le \frac{1}{ q_\us} - \frac{1}{2^{d_{1} -1} q_\us^2 } \le\frac{1}{q_{n_1}}-\frac{1}{q_{n_1}\cdots q_{n_{2}}} \le   \lb u_{n_1-1}x_S\rb,
     \end{equation}
     where the first inequality is obtain in analogy with (\ref{eqNew2}), while the last one follows from (\ref{eqNew1.5}) with $k=1$ .  Now the assertion
follows from (a), (\ref{eqNew3}), the definition of $\varrho_\us$ and the obvious inequality $\varrho_\us(x_S,0) \geq \max\{\|x_S\|, \sup\{\lb u_{n_k-1}x_S\rb: k \in \N \} \}$.

(c) It is easy to see that
 \begin{equation}\label{eqxsgen2} 
 x_S=\left( \frac{q_{n_1+1}-1}{u_{n_1+1}}+ \frac{q_{n_1+2}-1}{u_{n_1+2}} + \ldots + \frac{q_{n_2}-1}{u_{n_2}} \right) + \left( \frac{q_{n_3+1}-1}{u_{n_3+1}}
+ \ldots + \frac{q_{n_4}-1}{u_{n_4}}\right) + \ldots
  \end{equation}
Since infinitely many coefficients $c_j =0$ (e.g., those with $n_{2k} < j \leq n_{2k+1}$),  (\ref{eqxsgen2}) is the canonical representation of $x_S$. 
The same holds for $x_{S'}$. It is clear now, that $x_S\ne x_{S'}$ when $S \ne S'$.  \end{proof}

 \medskip\noindent{\bf Proof of Proposition C.} Pick a subsequence $S$ of $S_\us^* =(m_k)$ as in (\ref{eqxsgen1}), such that  $n_1 >m_1$, so that $u_{n_1} \ge 2q_\us$. Hence, 
 \begin{equation}\label{eqNew9} 
 \lb x_S\rb \le  \frac{1}{u_{n_1}} \le  \frac{1}{2q_\us}. 
   \end{equation}

From the previous lemma we have 
 \begin{equation}\label{eqNew} 
 \frac{1}{q_{\us}} - \frac{1}{2^{d_k+1 }}
%\frac{1}{q_{n_k}}-\frac{1}{q_{n_k}\cdots q_{n_{k+1}}} 
\le\lb u_{n_k-1}x_S\rb \le \frac{1}{q_{n_k}}-\frac{1}{q_{n_k}\cdots q_{n_{k+1}}}+\frac{1}{q_{n_k}\cdots q_{n_{k+2}}} \leq 
 \frac{1}{q_{\us}}-\frac{1}{q_\us^{d_k +1 }}+\frac{1}{2^{d_k+ d_{k+1}}}. 
   \end{equation}
   
   If the sequence of differences $\delta_k= m_{k+1}-m_k$ is unbounded, then we can choose 
   $S$ such that $(d_k)$ is divergent, then (\ref{eqNew9}), (\ref{eqNew}) and item (c) of the previous lemma ensure that $\varrho_\us(0,x_S)=\frac{1}{q_\us}$. Since there are $\cc$ many subsequences $S$ of $S_\us$ with this property, this proves item (a) in view of Lemma \ref{Lstima}(c). 

Now assume that the sequence $(\delta_k)$ is bounded.  Pick a subsequence $S$ of $S_\us$ as in (\ref{eqxsgen1}), such that $d_k = n_{k+1}-n_k\le d$ for some $d$ and for all $k$. From the equality in (\ref{eqNew2}) one has 
\begin{align*}\lb u_{n-1}x_S\rb&\le\frac{1}{q_{n_k}}-\frac{1}{q_{n_k}\cdots q_{n_{k+1}}}+\frac{1}{q_{n_k}\cdots q_{n_{k+2}}} = \frac{1}{q_\us} - \frac{1}{q_{n_k}\cdots q_{n_{k+1}}}\left(1-\frac{1}{q_{n_{k+1}+1}\cdots q_{n_{k+2}}}\right)\le
\\
 &\le \frac{1}{q_\us} - \frac{1}{q_\us^{d_k +1}}\left(1-\frac{1}{2}\right) 
  \le\frac{1}{q_\us} - \frac{1}{ 2 q_\us^{d} } < \frac{1}{q_\us}
%
% &=\frac{1}{q_{n_k}}\left(1-\frac{1}{q_{n_k+1}\cdots q_{n_{k+1}}}+\frac{1}{q_{n_k+1}\cdots q_{n_{k+2}}}\right)\\
% &=\frac{1}{q_{n_k}}\left(1-\frac{1}{q_{n_k+1}\cdots q_{n_{k+1}}}\left(1-\frac{1}{q_{n_{k+1}+1}\cdots q_{n_{k+2}}}\right)\right)\\
% &\le\frac{1}{C}\left(1-\frac{1}{C^{d_k}}\left(1-\frac{1}{C}\right)\right)\\
% &\le\frac{1}{C}\left(1-\frac{1}{C^{d}}\left(1-\frac{1}{C}\right)\right) =\frac{1}{C}-\frac{1}{C^{d+1}}+\frac{1}{C^{d+2}}< \frac{1}{C}.
\end{align*}

Hence, from (\ref{eqNew9}) and  in view of Lemma \ref{Lstima}(b) one can deduce $x_S\in B_{\frac{1}{q_\us}}^{\varrho_\us}(0)$.  To get $\left|B_{\frac{1}{q_\us}}^{\varrho_\us}(0)\right|=\cc$ take $\cc$ many subsequences $S$ of $S_\us$ as above and use item (c) of the above lemma. 
 \hfill$\Box$\medskip

 The following remark shows that the implications of Theorem \ref{TEgg} cannot be inverted 

\begin{rem}\label{B-M-W} \cite{BDMW1} If $H$ is an infinite characterized subgroup of $\T$, with characterizing sequence $\us\in\Z^\N$, then for every 
$m\in\N$ one can find a strictly increasing sequence $\vs \sim \us$ such that $q_\vs=\infty$ and $\liminf_n\frac{v_{n+1}}{v_n}=m$. 
%Let $\us\in\Z^\N$ be strictly increasing and . Then there is a such that $\tTu=\tus(\T)$, $\limsup_n\frac{u_{n+1}^*}{u_n^*}=\infty$ and $\liminf_n\frac{u_{n+1}^*}{u_n^*}=m$ (see \cite{BDMW1} for a proof). 
\end{rem}

This remark shows that the properties of having bounded ratios, or having ratios converging to $\infty$ are not $\sim$-invariant. Actually, if one takes a sequence $\us$ with bounded ratios, and then 
a sequence $\us^*\sim \us$ as in the remark, then one will have also $\us^*\approx \us$, according to Corollary A1, while $\us^*$ will fail to have the property of having bounded ratios.

% -------------------------------------------------------------------------------------------

\subsubsection{The torsion subgroup of $\tTu$}\label{Stors}
      
The polishability of $\tTu$ holds for a general integer sequence, but in the case of an a-sequence we can produce a simply defined dense 
countable subgroup witnessing the separability of $\tTu$. If $\us$ is an a-sequence, it makes sense to consider the subgroup $\ttTu$ of $\tTu$ formed by all $x\in\T$  with $|supp_\us(x)|<\infty$. We prove in Proposition \ref{Ptoraseq} that $\ttTu=t(\tTu)$ when $\us$ is an a-sequence and $t(\tTu)$ is the torsion subgroup of $\tTu$. \NB In other words, $supp_\us(x)$ is infinite for a torsion element $x$ of $\T$ if and only if $x\not \in \tu(\T)$. 
	
	\begin{prop}\label{Ptoraseq}  If $\us$ is an a-sequence, then $\ttTu=t(\tTu)$. \end{prop}

	\begin{proof}  The inclusion $\ttTu\subseteq t(\tTu)$ is obvious. Let $x=\frac{a}{b}\in t(\tTu)=\tTu\cap\Q/\Z$, where $a,b$ are coprime  and $a<b$. Hence, there exists $m\in\N$ such that $\forall n\ge m,\lb u_n x\rb<\frac{1}{b}$. In particular,
	 \begin{equation*}
	  \lb u_{m}\frac{a}{b}\rb<\frac{1}{b}\Leftrightarrow\lb u_{m}\frac{a}{b}\rb=0\Leftrightarrow b|u_ma.
	 \end{equation*}
 As $a$ and $b$ are coprime, we deduce that $u_m=lb$ for some $l\in\N$ and hence $x=\frac{a}{b}=\frac{al}{u_m}$. Hence, we have to show that elements of the form $x=\frac{k}{u_m}$ have a finite canonical representation. Obviously, we can assume that $k < u_m$. Hence, using the sequence $q_m = \frac{u_m}{u_{m-1}}, q_mq_{m-1} = \frac{u_m}{u_{m-2}}, \ldots,   q_mq_{m-1}\ldots q_{2} = \frac{u_m}{u_1}$ we can find integer coefficients $c_m, \ldots, c_1$ such that 
 \begin{equation}\label{eq2}
0\leq c_j < q_j \ \mbox{ for } \ 0 \leq j \leq m. 
 \end{equation}
and
 \begin{equation}\label{eq1}
k = c_m +  c_{m-1}  \frac{u_m}{u_{m-1}} +  c_{m-2} \frac{u_m}{u_{m-2}} + \ldots +  c_1 \frac{u_m}{u_{1}}
 \end{equation}
Dividing (\ref{eq1}) by $u_m$ we get 
 \begin{equation*}
\frac{k}{u_{m}}  =  \frac{c_m}{u_{m}} + \frac{c_{m-1}}{u_{m-1}} +   \frac{c_{m-2}}{u_{m-2}} + \ldots +   \frac{c_1}{u_{1}}.
 \end{equation*}
This provides a  finite canonical representation of $x = \frac{k}{u_{m}}$ in view of (\ref{eq2}). Therefore, $x\in \ttTu$. 
 	\end{proof}
		
\begin{notation}	
If $\PP$ denotes the set of all prime numbers, then for every $p\in\PP$, $\Z(p^\infty)$ and $\Z(p^n)$ denote respectively the Pr\"ufer $p$-group and the cyclic group of order $p^n$ where $n\in\N$. For $\us\in\Z^\N$ and $p\in\PP$, let $v_p$ be the additive $p$-adic valuation and
	\begin{equation*}
	 n_p(\us):=\liminf_{n\to\infty}{v_p(u_n)}\in\N\cup\{\infty\}.
	\end{equation*}

\end{notation}	

In  the following proposition  the torsion subgroup of the torus $t(\T)=\Q/\Z$ is identified with $\bigoplus_{p\in\PP}\Z(p^\infty)$.
	
\begin{prop}[\cite{BDMW}]\label{pTorStruc}
For $\us\in\Z^\N$ one has that \[t(\tTu)=\bigoplus_{p\in\PP}\Z(p^{n_p(\us)}).\]
\end{prop}

\medskip\noindent{\bf Proof of Proposition D.} 
We have to prove that the torsion subgroup $t(\tTu)$ of $\tTu$ is a dense subgroup of $(\tTu,\tau_\varrho)$ when $\us$ is an a-sequence. 
By Proposition \ref{Ptoraseq}, we have to prove that $\ttTu$ is a dense subgroup of $(\tTu,\tau_\varrho)$.
 
By Theorem \ref{bounded}, 
$\ttTu=\tTu$, when $q_\us<\infty$. Hence, there is nothing to prove in this case. 

Suppose that $q_\us=\infty$, hence there exists a subsequence $q_{n_k}\to\infty$. Let $x\in\tTu$ with canonical representation (\ref{canonical}) and $\ep>0$. It suffices to find $x'\in\ttTu$ such that $x-x'\in B_\ep^\varrho(0)$.

Let $k^*\in\N$ such that:
\begin{itemize}
 \item[(i)] $\forall k\ge k^*,q_{n_k}>\frac{2}{\ep}$;
 \item[(ii)]$\forall n\ge n_{k^*}, \lb u_nx\rb< \frac{\ep}{2}$.
\end{itemize}
If 
\begin{equation*}
 x'=\sum_{n\le n_{k^*}}\frac{c_n}{u_n} \quad \text{ and } \quad x''=\sum_{n>n_{k^*}}\frac{c_n}{u_n},
\end{equation*}
then $x=x'+x''$ and $x'\in\ttTu$. Hence it is enough to prove that
% $x''\in B_\ep^\varrho(0)$, i.e., 
$\lb u_sx''\rb<\ep$ holds for all $s\in\N$.

If $s\ge n_{k^*}$, then $\lb u_s x''\rb=\lb u_s(x-x')\rb\le\lb u_sx\rb+\lb u_s x'\rb<\ep+0$.

If $s<n_{k^*}$, then
\begin{align*}
 \lb u_sx''\rb\le u_sx''&=\sum_{n>n_{k^*}}\frac{c_nu_s}{u_n}\\&=\frac{1}{q_{s+1}\cdots q_{n_{k^*}}}\left(\frac{c_{n_{k^*}+1}}{q_{n_{k^*}+1}}+\cdots+\frac{c_{n_{k^*}+t}}{q_{n_{k^*}+1}\cdots q_{n_{k^*}+t}}+\cdots\right).
\end{align*}
Since $(c_n)$ is a canonical representation, $\left(\frac{c_{n_{k^*}+1}}{q_{n_{k^*}+1}}+\cdots+\frac{c_{n_{k^*}+t}}{q_{n_{k^*}+1}\cdots q_{n_{k^*}+t}}+\cdots\right)<1$ holds, according to Remark \ref{NewRem}. By (i), $\frac{1}{q_{s+1}\cdots q_{n_{k^*}}}<\frac{\ep}{2}$ and hence $\lb u_sx''\rb<\ep$.
\hfill$\Box$\medskip

%%%%%%%%%%%%%%%%%%%%%%%%%%%%%%%%%%%%%
% ------------------  P r o o f   o f   T h e o r e m s   D   a n d  E   ------------------
%%%%%%%%%%%%%%%%%%%%%%%%%%%%%%%%%%%%%%

\subsubsection{Proof of Theorem E}

\begin{prop}\label{LemmaX} If $\us$ is an a-sequence with $ \tTu\in\tau_{\us}$, then the sequence of ratios $(q_n)$ is bounded. 
\end{prop}

\begin{proof}
  Assume for a contradiction that
 \begin{equation}\label{assmption}
\lim _t q_{n_t} =\infty
 \end{equation}
for some subsequence \NB $\{q_{n_t}:t\in \N\}$. We can assume without loss of generality that $q_{n_t}\geq 2$ for all $t$. We prove that $ \tTu\notin\tau_{\us}$. 
To do that we prove that  $B_\ep^\varrho(0)\nsubseteq\tTu$ for every $\ep< \frac{1}{2}$. Hence we need an element $x\in B_\ep^\varrho(0)$ such that $x\notin\tTu$.

Let 
 \begin{equation}\label{(dag)}
x=\sum_{n}\frac{c_n}{u_n};
 \end{equation}
 where $c_n=0$ if $n\neq n_t$ for every $t$ otherwise if $n=n_t$ $c_n=\lfloor q_n\ep/4\rfloor$. Note that (\ref{(dag)}) is a canonical representation of $x$, as 
 \begin{equation}\label{LAST}
c_{n_t}\le \frac{q_{n_t}\ep}{4} <\frac{q_{n_t}}{8} <q_{n_t}-1
 \end{equation}
for every $t$.

Let us prove that $x\in B^\varrho_{\varepsilon}(0)$. Pick $s\in\N$ and find a  positive $r\in\N$ such that $n_{r-1}\le s\le n_{r}-1$.

Hence, one has the following congruence modulo 1
$$ 
u_sx \equiv_1 \frac{c_{s+1}}{q_{s+1}}+\frac{c_{s+2}}{q_{s+1}q_{s+2}}+\cdots+\frac{c_{n_r-1}}{q_{s+1}\cdots q_{n_r-1}}+\frac{c_{n_r}}{q_{s+1}\cdots q_{n_r}}+\cdots.$$

By construction $c_{s+1}=\cdots=c_{n_r-1}=0$ and $c_{n_j}\le \frac{q_{n_j}\ep}{4}$ 
for all $j \geq r$, according to (\ref{LAST}).  Therefore, $\frac{c_{n_j}}{q_{n_j}} \leq \frac{\ep}{4}$, hence $\frac{c_{n_j}}{q_{s+1}\cdots q_{n_j}}\leq \frac{\ep}{2}\cdot \frac{1}{2^{n_j-s}}$. This yields
$$ 
\lb u_sx\rb= \lb\frac{c_{n_r}}{q_{s+1}\cdots q_{n_r}}+\cdots\rb\le\frac{\ep}{2}\sum_{j=r}^\infty\frac{1}{2^{n_j-s}}\le\frac{\ep}{2}\sum_{m=1}^\infty\frac{1}{2^{m}}=\frac{\ep}{2}.$$

Hence for every $s$ one has $\lb u_sx\rb\le\frac{\ep}{2}<\ep$, i.e., $x\in B^\varrho_{\varepsilon}(0)$.

\smallskip 

 As $ \frac{q_{n_t}\ep}{4} -1\le c_{n_t}\le \frac{q_{n_t}\ep}{4},$ the following inequality holds  $$\frac{\ep}{4}-\frac{1}{q_{n_t}}\le\frac{c_{n_t}}{q_{n_t}}\le\frac{\ep}{4}.$$ 

Therefore, (\ref{assmption}) implies that $\frac{c_{n_t}}{q_{n_t}}$ converges to $\frac{\ep}{4}\neq 0$. By Corollary \ref{unbounded**}, this yields $x\notin\tTu$. \end{proof}

 The above proposition fails  in the case of a general sequence of integers. Indeed, if $H$ is a countable subgroup of $\T$, then 
 $H$ is characterized by some sequence $\us$ that need not have bounded ratios (in case  $\us$ has 
 bounded ratios, we can replace $\us$ by some $\us^*\sim \us$ with $q_{\us^*}= \infty$, according to 
 Remark \ref{B-M-W}). 
 
% using the procedure of Remark \ref{B-M-W}, any $q$-bounded a-sequence can be turned on an unbounded sequence (that is necessarily not arithmetic, see Remark \ref{rIrr}) with the same characterized subgroup. The new sequence is the desired counterexample.
%\end{rem}

%\medskip\noindent{\bf Proof of Theorem E.}
%(i)$\Leftrightarrow$(ii) since, $\tTu$ is Borel;
%
%(ii)$\Leftrightarrow$(iii) is Corollary C2 for $X=\T$;
%
%(ii)$\Leftrightarrow$(iv) by Theorem D.
%
%Moreover if $\left(\frac{u_n}{u_{n-1}}\right)$ is bounded, then by Theorem \ref{TEgg} $\tTu$ is countable.

%\hfill$\Box$\medskip

\medskip\noindent{\bf Proof of Theorem F.}  The implication (a)$\Rightarrow$(b) follows by Theorem \ref{bounded}; while (b)$\Rightarrow$(c)$\Rightarrow$(d) are obvious.

(d)$\Rightarrow$(e) follows by Theorem A, 

(e)$\Rightarrow$(f) follows from  the fact that $\tau_\us^*\subseteq\tau_\us$;

(f)$\Rightarrow$(a) follows from  Proposition \ref{LemmaX}. 
\hfill$\Box$\medskip
%%%%%%%%%%%%%%%%%%%%%%%%%%%%%%%%%%%%%%%%%%%%%%%%%%%%%%%%%%%%%%%%%%%%%%%%%%%%%%%%%%%%%%%%%%%%%%%%%%%%%%%%%%%%%%%%%%

\end{document}